\documentclass[10pt]{article}
\usepackage{epsf}

\usepackage{amsthm}
\usepackage{amssymb}
\usepackage{amsmath}

\usepackage[top=1.5cm,bottom=1.5cm,left=2.5cm,right=2.5cm,includehead,includefoot
           ]{geometry}

\makeatletter

\numberwithin{equation}{section}
\newtheorem{theorem}{Theorem}[section]
\newtheorem{lemma}[theorem]{Lemma}
\newtheorem{corollary}[theorem]{Corollary}
\newtheorem{remark}[theorem]{Remark}

\newtheorem{proposition}[theorem]{Proposition}

\title{On (m, n)-derivations of Some Algebras}
\author{\begin{tabular}{c} Jiankui Li\footnote{Corresponding author.
E-mail address: jiankuili@yahoo.com},~ Qihua Shen~and Jianbin Guo\\
{\small\it Department of Mathematics, East China University of
Science and Technology}\\
{\small\it Shanghai 200237, P. R. China}
\end{tabular}}
\date{}
\begin{document}

\maketitle \abstract
Let $\mathcal{A}$ be a unital algebra, $\delta$ be a linear mapping from $\mathcal{A}$ into itself and $m$, $n$ be fixed integers.
We call $\delta$ an (\textit{m, n})-derivable mapping at $Z$, if
$m\delta(AB)+n\delta(BA)=m\delta(A)B+mA\delta(B)+n\delta(B)A+nB\delta(A)$ for all $A, B\in \mathcal{A}$ with $AB=Z$.
In this paper, (\textit{m, n})-derivable mappings at $0$ (resp. $I_\mathcal{A}\oplus0$, $I$)
on generalized matrix algebras are characterized. We also study (\textit{m, n})-derivable mappings at $0$ on CSL algebras.
We reveal the relationship between this kind of mappings with Lie derivations, Jordan derivations and derivations.

\


\

{\sl 2000 AMS classification} : Primary 47L35; Secondly 16W25
\

\section{Introduction}\

Let $\mathcal{R}$ be a unital ring and $\mathcal{A}$ be a unital $\mathcal{R}$-algebra.
Let $\delta$ be a linear mapping from $\mathcal{A}$ into itself.
We call $\delta$ a \textit{derivation} if $\delta(AB)=\delta(A)B+A\delta(B)$ for all $A, B\in \mathcal{A}$.
We call $\delta$ a \textit{Jordan derivation} if $\delta(AB+BA)=\delta(A)B+A\delta(B)+\delta(B)A+B\delta(A)$ for all $A, B\in \mathcal{A}$.
$\delta$ is called a \textit{Lie derivation} if $\delta([A, B])=[\delta(A),B]+[A,\delta(B)]$ for all $A, B\in \mathcal{A}$, where $[A, B]=AB-BA$.
The questions of characterizing Jordan derivations and Lie derivations have received considerable attention from several authors, who revealed the relationship
between Jordan derivations, derivations as well as Lie derivations (for example, \cite{SALD, semiprime, LDTA, prime, NON} and the references therein).

Let $m$, $n$ be fixed integers. In \cite{mn}, Vukman defined a new type of Jordan derivations, named (\textit{m, n})-\textit{Jordan derivation},
that is, an additive mapping $\eta$ from a ring $\mathcal{R}$ into itself such that
$(m+n)\eta(A^2)=2m\eta(A)A+2nA\eta(A)$ for every $A\in \mathcal{R}$.
He proved that each (\textit{m, n})-Jordan derivation of a prime ring is a derivation.
Motivated by this, we define a new type of derivations, named (\textit{m, n})-\textit{derivation}.
An (\textit{m, n})-derivation is a linear mapping $\delta$ from $\mathcal{A}$ into itself such that
$m\delta(AB)+n\delta(BA)=m\delta(A)B+mA\delta(B)+n\delta(B)A+nB\delta(A)$ for all $A, B\in \mathcal{A}$.
Obviously, every (1, 1)-derivation is a Jordan derivation, each (1, $-$1)-derivation is a Lie derivation,
(1, 0)-derivations and (0, 1)-derivations are derivations.

Recently, there have been a number of papers on the study of conditions under which derivations
on algebras can be completely determined by their action on some subsets of elements. Let $\delta$ be a linear mapping from $\mathcal{A}$ into itself and $Z$ be in $\mathcal{A}$.
$\delta$~is called~\textit{derivable at ~$Z$~},~if $\delta(AB)=\delta(A)B+A\delta(B)$ for all $A,B\in \mathcal{A}$~with~$AB=Z$; $\delta$~is called~\textit{Jordan derivable at ~$Z$~},~if $\delta(AB+BA)=\delta(A)B+A\delta(B)+\delta(B)A+B\delta(A)$ for all ~$A,B\in \mathcal{A}$~with~$AB=Z$;
$\delta$~is called~\textit{Lie derivable at ~$Z$~},~if $\delta([A,B])=[\delta(A),B]+[A,\delta(B)]$ for all~$A,B\in \mathcal{A}$~with~$AB=Z$. 
It is natural and interesting to ask whether or not a linear mapping is a derivation (Jordan derivation, or Lie derivation) if it is derivable (Jordan derivable, Lie derivable) only at one given point. An and Hou \cite{ANHOU} investigated derivable mappings at $0$, $P,$ and $I$ on triangular rings, where $P$ is some fixed non-trivial idempotent. Let $X$ be a Banach space, Lu and Jing \cite{COLD} studied Lie derivable mappings at $0$ and $P$ on $B(X)$, where $P$ is a fixed nontrivial idempotent. In \cite{ZHUJUN}, Zhao and Zhu characterized Jordan derivable mappings at $0$ and $I$ on triangular algebras.
(see \cite{AnHou, CHDM,Hou6, Hou4,Hou3, APM, LI7,LI5, COLD, LU1,Hou2,ZHUJUN, Zhu3,Zhu4,ZHUJUN2} and the references therein for more related results). Motivated by these facts, we introduce the concept of the mappings (m, n)-derivable at some point. We say that a linear mapping $\delta$ from $\mathcal{A}$ into itself is an (\textit{m, n})-\textit{derivable mapping} at $Z$, if $m\delta(AB)+n\delta(BA)=m\delta(A)B+mA\delta(B)+n\delta(B)A+nB\delta(A)$ for all $A, B\in \mathcal{A}$ with $AB=Z$.
In this paper, we study this kind of mappings on generalized matrix algebras and CSL algebras.

Let $\mathcal{R}$ be a unital ring. A \textit{Morita context} is a set $(\mathcal{A}, \mathcal{B}, \mathcal{M}, \mathcal{N})$
and two mappings $\phi$ and $\varphi$, where $\mathcal{A}$ and $\mathcal{B}$ are two $\mathcal{R}$-algebras, $\mathcal{M}$ is an
$(\mathcal{A}, \mathcal{B})$-bimodule, $\mathcal{N}$ is a $(\mathcal{B}, \mathcal{A})$-bimodule, and mappings $\phi: \mathcal{M}\otimes_\mathcal{B} \mathcal{N}\rightarrow \mathcal{A}$
and $\varphi: \mathcal{N}\otimes_\mathcal{A} \mathcal{M}\rightarrow \mathcal{B}$ are two bimodule homomorphisms satisfying the following commutative diagrams:

\setlength{\unitlength}{1cm}
\begin{picture}(4,4)
\put(1,3){$\mathcal{M}\otimes_\mathcal{B} \mathcal{N}\otimes_\mathcal{A} \mathcal{M}$} \put(3.4,3.05){\vector(1,0){2.2}}
\put(1.5,0.5){$\mathcal{M}\otimes_\mathcal{B} \mathcal{B}$} \put(2.8,0.55){\vector(1,0){3.1}}
\put(5.8,3){$\mathcal{A} \otimes_\mathcal{A} \mathcal{M}$} \put(6.3,2.75){\vector(0,-1){1.8}}
\put(6.1,0.5){$\mathcal{M}$} \put(2,2.75){\vector(0,-1){1.8}}
\put(4.2,0.65){$\cong$} \put(6.4,1.7){$\cong$}
\put(4.1,3.2){$\phi \otimes I_\mathcal{M}$}
\put(2.1,1.7){$I_\mathcal{M} \otimes \varphi$}
\put(-0.04,-0.5){and}
\end{picture}

\begin{picture}(4,4)
\put(1,3){$\mathcal{N}\otimes_\mathcal{A} \mathcal{M}\otimes_\mathcal{B} \mathcal{N}$} \put(3.4,3.05){\vector(1,0){2.2}}
\put(1.5,0.5){$\mathcal{N}\otimes_\mathcal{A} \mathcal{A}$} \put(2.8,0.55){\vector(1,0){3.1}}
\put(5.8,3){$\mathcal{B} \otimes_\mathcal{B} \mathcal{N}$} \put(6.3,2.75){\vector(0,-1){1.8}}
\put(6.1,0.5){$\mathcal{N}$} \put(2,2.75){\vector(0,-1){1.8}}
\put(4.2,0.65){$\cong$} \put(6.4,1.7){$\cong$}
\put(4.1,3.2){$\varphi \otimes I_\mathcal{N}$}
\put(2.1,1.7){$I_\mathcal{N} \otimes \phi$}
\put(6.5,0.48){.}
\end{picture}

These conditions insure that the set
$$\left[
  \begin{array}{cc}
    \mathcal{A} & \mathcal{M} \\
    \mathcal{N} & \mathcal{B} \\
  \end{array}
\right]=\left\{\left[
            \begin{array}{cc}
              A& M \\
              N & B \\
            \end{array}
          \right]: ~A\in \mathcal{A}, M\in \mathcal{M}, N\in \mathcal{N}, B\in \mathcal{B}\right\}$$
forms an $\mathcal{R}$-algebra under matrix-like addition and matrix-like multiplication if we put $MN=\phi(M,N)$ and $NM=\varphi(N,M)$.
We call such an $\mathcal{R}$-algebra a \textit{generalized matrix algebra} and denote it by
$\mathcal{U}=\left[
  \begin{array}{cc}
    \mathcal{A} & \mathcal{M} \\
    \mathcal{N} & \mathcal{B} \\
  \end{array}
\right]$, where $\mathcal{A}$ and $\mathcal{B}$ are two unital algebras and at least one of the two bimodules $\mathcal{M}$ and $\mathcal{N}$
is distinct from zero. Note that different choices of pairs of bimodule homomorphisms generally lead up to different algebras, even if we have the same set $(\mathcal{A}, \mathcal{B}, \mathcal{M}, \mathcal{N})$. This kind of algebra was first introduced by Sands in \cite{sands}.
Obviously, when $\mathcal{M}=0$ or $\mathcal{N}=0$, $\mathcal{U}$ degenerates to the triangular algebra.
We denote $I_\mathcal{A}$ the unit element in $\mathcal{A}$, $I_\mathcal{B}$ the unit element in $\mathcal{B}$ and $A\oplus B$ the element $\left[
  \begin{array}{cc}
    A & 0 \\
    0 & B \\
  \end{array}
\right]$ in $\mathcal{U}$.
Since $(I_\mathcal{A} \oplus 0) \mathcal{U} (I_\mathcal{A} \oplus 0)$ is a subalgebra of $\mathcal{U}$ isomorphic to $\mathcal{A}$, we will make no difference between the notations $(I_\mathcal{A} \oplus 0) \mathcal{U} (I_\mathcal{A} \oplus 0)$ and $\mathcal{A}$. Similarly, we regard $(0 \oplus I_\mathcal{B}) \mathcal{U} (0 \oplus I_\mathcal{B})$ the same as $\mathcal{B}$.
In Section 2-4, we shall show that if $\delta$ is an ($m, n$)-derivable mapping at 0
(resp. $I_\mathcal{A}\oplus0$, $I$) from $\mathcal{U}$ into itself,
then $\delta$ is a derivation, a Jordan derivation or a Lie derivation according to different choices of $m$ and $n$.

Let $H$ be a separable complex Hilbert space and $B(H)$ be the set of all bounded linear operators from $H$ into itself.
By a \textit{subspace lattice} on $H$, we mean a collection $\mathcal{L}$ of closed subspaces of $H$
with (0) and $H$ in $\mathcal{L}$ such that for every family $\{M_r\}$ of elements of $\mathcal{L}$,
both $\cap M_r$ and $\vee M_r$ belong to $\mathcal{L}$. For a subspace lattice $\mathcal{L}$ of $H$,
let alg$\mathcal{L}$ denote the algebra of all operators in $B(H)$ that leave members of $\mathcal{L}$ invariant;
and for a subalgebra $\mathcal{A}$ of $B(H)$, let lat$\mathcal{A}$ denote the lattice of all closed subspaces of $H$
that are invariant under all operators in $\mathcal{A}$.
An algebra $\mathcal{A}$ is called \textit{reflexive} if alglat$\mathcal{A}=\mathcal{A}$;
and dually, a subspace lattice $\mathcal{L}$ is said to be \textit{reflexive} if latalg$\mathcal{L}=\mathcal{L}$.
Every reflexive algebra is of the form alg$\mathcal{L}$ for some subspace lattice $\mathcal{L}$ and vice versa.
For convenience, we disregard the distinction between a closed subspace and the orthogonal projection onto it.
A totally ordered subspace lattice $\mathcal{N}$ on $H$ is called a \textit{nest}
and the corresponding algebra alg$\mathcal{N}$ is called a \textit{nest algebra}.
As an immediate but noteworthy application of the results in Section 2-4,
we characterize the mappings (\textit{m, n})-derivable at 0 (resp. $P$, $I$) from alg$\mathcal{N}$ into itself.
A subspace lattice $\mathcal L$ on $H$ is called a \textit{commutative~subspace~lattice}
(or \textit{CSL} for short), if all projections in $\mathcal L$ commute pairwise. If $\mathcal L$ is a CSL, then
alg$\mathcal{L}$ is called a \textit{CSL algebra}. By \cite{CSL}, we know that if $\mathcal{L}$ is a CSL, then $\mathcal{L}$
is reflexive. In Section 5, we show that if $\delta$ is a norm-continuous ($m, n$)-derivable mapping at 0 on CSL algebras with $m+n\neq 0$ and $\delta(I)=0$,
then $\delta$ is a derivation.

We call $\mathcal{M}$ a \textit{faithful left $\mathcal{A}$-module} if for any $A\in \mathcal{A}$, $A\mathcal{M}=0$ implies
$A=0$. Similarly, we can define a \textit{faithful right $\mathcal{B}$-module}.
If $\mathcal{M}$ is a faithful left $\mathcal{A}$-module and a faithful right $\mathcal{B}$-module,
then $\mathcal{M}$ is called a \textit{faithful $(\mathcal{A}, \mathcal{B})$-bimodule}.
Given an integer $n\geq 2$, we say that the characteristic of an algebra $\mathcal{A}$ is not $n$, if for every $A\in \mathcal{A}$, $nA=0$ implies $A=0$.

\section{(\textit{m, n})-derivable mappings at 0}
\

In this section, we study (\textit{m, n})-derivable mappings at $0$ on generalized matrix algebras. In the following of this paper, we always assume that $m\neq 0$ and $n\neq 0$.
\begin{theorem}\label{201}
Let $\mathcal{U}=\left[
  \begin{array}{cc}
    \mathcal{A} & \mathcal{M} \\
    \mathcal{N} & \mathcal{B} \\
  \end{array}
\right]$ be a generalized matrix algebra and $\delta$ be an $(m, n)$-derivable mapping at 0 from $\mathcal{U}$ into itself.
Assume that $\mathcal{M}$ is a faithful $(\mathcal{A}, \mathcal{B})$-bimodule.

$\mathrm{(1) }$~~~If $(m+n)(m-n)\neq 0$ and $char(\mathcal{U})\neq |mn(m+n)(m-n)|$, then $\delta$ is a derivation.

$\mathrm{(2)}$~~~If $m+n=0$,
$\delta([\mathcal{A},\mathcal{A}])\bigcap \mathcal{B}=0$, $\delta([\mathcal{B},\mathcal{B}])\bigcap \mathcal{A}=0$ and $char(\mathcal{U})\neq |2m|$,
then $\delta$ is a Lie derivation.

$\mathrm{(3)}$~~~If $m-n=0$ and $char(\mathcal{U})\neq |2m|$, then $\delta$ is a Jordan derivation.
\end{theorem}

Since $\delta$ is linear, for any $A\in \mathcal{A}$, $M\in \mathcal{M}$, $N\in \mathcal{N}$ and $B\in \mathcal{B}$,
we may write
 \begin{eqnarray*}
\delta\left(\left[\begin{array}{cc}A & M \\N & B \\\end{array}\right]\right)
=\left[\begin{array}{cc}a_{11}(A)+b_{11}(B)+c_{11}(M)+d_{11}(N) & a_{12}(A)+b_{12}(B)+c_{12}(M)+d_{12}(N) \\a_{21}(A)+b_{21}(B)+c_{21}(M)+d_{21}(N) & a_{22}(A)+b_{22}(B)+c_{22}(M)+d_{22}(N) \\\end{array}\right],
\end{eqnarray*}
where $a_{ij}$, $b_{ij}$, $c_{ij}$ and $d_{ij}$ are linear mappings, $i, j\in \{1, 2\}$.

To prove Theorem \ref{201}, we first show a lemma and several propositions.

\begin{lemma}\label{202}
Let $\mathcal{U}=\left[
  \begin{array}{cc}
    \mathcal{A} & \mathcal{M} \\
    \mathcal{N} & \mathcal{B} \\
  \end{array}
\right]$ be a generalized matrix algebra with $char(\mathcal{U})\neq |mn|$ and $\delta$ be an $(m, n)$-derivable mapping at 0 from $\mathcal{U}$ into itself.
Then
$$\delta\left(\left[\begin{array}{cc}A & M \\N & B \\\end{array}\right]\right)=
\left[\begin{array}{cc}a_{11}(A)+b_{11}(B)-MN_0-M_0N & AM_0-M_0B+c_{12}(M)+d_{12}(N)
\\N_0A-BN_0+c_{21}(M)+d_{21}(N) & a_{22}(A)+b_{22}(B)+N_0M+NM_0 \\\end{array}\right],$$
where $N_0\in\mathcal{N}$, $M_0\in\mathcal{M}$ and $a_{11}:\mathcal{A}\rightarrow\mathcal{A}$,
$a_{22}:\mathcal{A}\rightarrow\mathcal{B}$, $b_{11}:\mathcal{B}\rightarrow\mathcal{A}$,
$b_{22}:\mathcal{B}\rightarrow\mathcal{B}$, $c_{12}:\mathcal{M}\rightarrow\mathcal{M}$,
$c_{21}:\mathcal{M}\rightarrow\mathcal{N}$, $d_{12}:\mathcal{N}\rightarrow\mathcal{M}$,
$d_{21}:\mathcal{N}\rightarrow\mathcal{N}$ are linear mappings satisfying
\begin{eqnarray*}
\mathrm{(i)}&&nc_{12}(AM)=mMa_{22}(A)+na_{11}(A)M+nAc_{12}(M),~nc_{21}(AM)=mc_{21}(M)A,\\
&&nd_{21}(NA)=ma_{22}(A)N+nNa_{11}(A)+nd_{21}(N)A,~nd_{12}(NA)=mAd_{12}(N),\\
&&mBa_{22}(A)+na_{22}(A)B=0;\\
\mathrm{(ii)}&&nc_{12}(MB)=mb_{11}(B)M+nMb_{22}(B)+nc_{12}(M)B,~nc_{21}(MB)=mBc_{21}(M),\\
&&nd_{21}(BN)=mNb_{11}(B)+nb_{22}(B)N+nBd_{21}(N),~nd_{12}(BN)=md_{12}(N)B,\\
&&mb_{11}(B)A+nAb_{11}(B)=0;\\
\mathrm{(iii)}&&a_{11}(MN)=c_{12}(M)N+Md_{21}(N)+b_{11}(NM),~b_{22}(NM)=Nc_{12}(M)+d_{21}(N)M+a_{22}(MN).
\end{eqnarray*}
\end{lemma}
\begin{proof}
We prove the lemma by two steps.\

\textbf{Step 1:}~~
For any $A\in\mathcal{A}$, $B\in\mathcal{B}$ and $M\in\mathcal{M}$,
let $S=\left[\begin{array}{cc}0 & M \\0 & B \\\end{array}\right]$
and $T=\left[\begin{array}{cc}A & 0 \\0 & 0 \\\end{array}\right]$. Then $ST=0$ and we have
\begin{eqnarray*}
&&n\left[\begin{array}{cc} c_{11}(AM) & c_{12}(AM) \\c_{21}(AM) & c_{22}(AM)\\\end{array}\right]
=m\delta(ST)+n\delta(TS)=m\delta(S)T+mS\delta(T)+n\delta(T)S+nT\delta(S)\\
&&~~~=m\left[\begin{array}{cc} b_{11}(B)+c_{11}(M) & b_{12}(B)+c_{12}(M) \\  b_{21}(B)+c_{21}(M)
& b_{22}(B)+c_{22}(M) \end{array}\right] \left[\begin{array}{cc}A & 0 \\0 & 0\\\end{array}\right]
+m\left[\begin{array}{cc} 0 & M \\  0 & B  \end{array}\right] \left[\begin{array}{cc}a_{11}(A) & a_{12}(A) \\a_{21}(A) & a_{22}(A)\\\end{array}\right]\\
&&~~~~~~+n\left[\begin{array}{cc}a_{11}(A) & a_{12}(A) \\a_{21}(A) & a_{22}(A)\\\end{array}\right] \left[\begin{array}{cc} 0 & M \\  0 & B  \end{array}\right]
+n\left[\begin{array}{cc}A & 0 \\0 & 0\\\end{array}\right]\left[\begin{array}{cc} b_{11}(B)+c_{11}(M)
& b_{12}(B)+c_{12}(M) \\  b_{21}(B)+c_{21}(M) & b_{22}(B)+c_{22}(M)  \end{array}\right].
\end{eqnarray*}
The above matrix equation implies the following four equations
\begin{eqnarray}
&&nc_{11}(AM)=mb_{11}(B)A+mc_{11}(M)A+mMa_{21}(A)+nAb_{11}(B)+nAc_{11}(M),\label{0001}\\
&&nc_{12}(AM)=mMa_{22}(A)+na_{11}(A)M+na_{12}(A)B+nAb_{12}(B)+nAc_{12}(M),\label{0002}\\
&&nc_{21}(AM)=mb_{21}(B)A+mc_{21}(M)A+mBa_{21}(A),\label{0003}\\
&&nc_{22}(AM)=mBa_{22}(A)+na_{21}(A)M+na_{22}(A)B.\label{0004}
\end{eqnarray}

Taking $B=0$ in (\ref{0002}) and (\ref{0003}), we have
\begin{eqnarray}
&&nc_{12}(AM)=mMa_{22}(A)+na_{11}(A)M+nAc_{12}(M),\label{0005}\\
&&~~~~~~~~~nc_{21}(AM)=mc_{21}(M)A\label{0006}
\end{eqnarray}
for all $A\in \mathcal{A}$ and $M\in \mathcal{M}$.

Taking $M=0$ in (\ref{0003}), we have
\begin{eqnarray}
b_{21}(B)A=-Ba_{21}(A)\label{0007}
\end{eqnarray}
for all $A\in \mathcal{A}$ and $B\in \mathcal{B}$.
Taking $A=I_\mathcal{A}$ and $B=I_\mathcal{B}$ in (\ref{0007}) gives $a_{21}(I_\mathcal{A})=-b_{21}(I_\mathcal{B})$.
Let $N_0=a_{21}(I_\mathcal{A})$. Then taking $A=I_\mathcal{A}$ and $B=I_\mathcal{B}$ in (\ref{0007}), respectively,
leads to
\begin{eqnarray}
a_{21}(A)=N_0A~~~\textrm{and}~~~b_{21}(B)=-BN_0\label{00081}
\end{eqnarray}
for every $A\in \mathcal{A}$ and every $B\in \mathcal{B}$.
Similarly, by taking $M=0$ in (\ref{0002}), we obtain $a_{12}(A)B=-Ab_{12}(B)$. Let $a_{12}(I_\mathcal{A})=M_0$,
and we have
\begin{eqnarray}
a_{12}(A)=AM_0~~~\textrm{and}~~~b_{12}(B)=-M_0B\label{00082}
\end{eqnarray}
for every $A\in \mathcal{A}$ and every $B\in \mathcal{B}$.

Taking $M=0$ in (\ref{0001}) and (\ref{0004}), we have
\begin{eqnarray}
mb_{11}(B)A+nAb_{11}(B)=0~~~\textrm{and}~~~mBa_{22}(A)+na_{22}(A)B=0
\end{eqnarray}
for all $A\in \mathcal{A}$ and $B\in \mathcal{B}$.

Taking $B=0$ and $A=I_\mathcal{A}$ in (\ref{0001}), as well as in (\ref{0004}), we obtain
\begin{eqnarray}
c_{11}(M)=-MN_0~~~\textrm{and}~~~c_{22}(M)=N_0M
\end{eqnarray}
for every $M\in \mathcal{M}$.

Similarly, by considering $S=\left[\begin{array}{cc}0 & 0 \\0 & B \\\end{array}\right]$
and $T=\left[\begin{array}{cc} 0 & M \\0 & 0 \\\end{array}\right]$, we arrive at
\begin{eqnarray}
&&nc_{12}(MB)=mb_{11}(B)M+nMb_{22}(B)+nc_{12}(M)B,\label{0008}\\
&&nc_{21}(MB)=mBc_{21}(M)\label{0009}
\end{eqnarray}
for all $B\in \mathcal{B}$ and $M\in\mathcal{M}$.

By considering $S=\left[\begin{array}{cc}0 & 0 \\N & 0 \\\end{array}\right]$ and
$T=\left[\begin{array}{cc}0 & 0 \\0 & B \\\end{array}\right]$, we have
\begin{eqnarray}
&&~~~~~~~~~nd_{21}(BN)=mNb_{11}(B)+nb_{22}(B)N+nBd_{21}(N),\label{0010}\\
&&d_{11}(N)=-M_0N,~~~d_{22}(N)=NM_0~~~ \textrm{and}~~~nd_{12}(BN)=md_{12}(N)B\label{0012}
\end{eqnarray}
for all $B\in \mathcal{B}$ and $N\in\mathcal{N}$.

By considering $S=\left[\begin{array}{cc}A & 0 \\0 & 0 \\\end{array}\right]$ and
$T=\left[\begin{array}{cc}0 & 0 \\N & 0 \\\end{array}\right]$,
we obtain
\begin{eqnarray}
&&nd_{21}(NA)=ma_{22}(A)N+nNa_{11}(A)+nd_{21}(N)A,\label{0013}\\
&&~~~~~~~~~~nd_{12}(NA)=mAd_{12}(N)\label{0014}
\end{eqnarray}
for all $A\in \mathcal{A}$ and $N\in\mathcal{N}$.

\textbf{Step 2:}~~
For any $N\in\mathcal{N}$ and $M\in\mathcal{M}$,
let $S=\left[\begin{array}{cc}-MN & M \\0 & 0 \\\end{array}\right]$ and
$T=\left[\begin{array}{cc}I_\mathcal{A} & 0 \\N & 0 \\\end{array}\right]$. Then $ST=0$ and we have\\
\begin{eqnarray*}
&&n\left[\begin{array}{cc}-a_{11}(MN)-MN_0+M_0NMN+b_{11}(NM) & -MNM_0+c_{12}(M)-d_{12}(NMN)-M_0NM
\\-N_0MN+c_{21}(M)-d_{21}(NMN)-NMN_0 & -a_{22}(MN)+N_0M-NMNM_0+b_{22}(NM) \\\end{array}\right]\\
&&~~~~~~~~~~~~~~~~=m\delta(ST)+n\delta(TS)=m\delta(S)T+mS\delta(T)+n\delta(T)S+nT\delta(S)\\
&&~~~~~~~~~~~~~~~~=m\left[\begin{array}{cc} -a_{11}(MN)-MN_0 & -MNM_0+c_{12}(M) \\ -N_0MN+c_{21}(M) & -a_{22}(MN)+N_0M \end{array}\right]
\left[\begin{array}{cc}I_\mathcal{A} & 0 \\N & 0\\\end{array}\right]\\
&&~~~~~~~~~~~~~~~~~~~+m\left[\begin{array}{cc} -MN & M \\  0 & 0  \end{array}\right]
\left[\begin{array}{cc}a_{11}(I_\mathcal{A})-M_0N & M_0+d_{12}(N) \\ N_0+d_{21}(N) & a_{22}(I_\mathcal{A})+NM_0\\\end{array}\right]\\
&&~~~~~~~~~~~~~~~~~~~+n\left[\begin{array}{cc}a_{11}(I_\mathcal{A})-M_0N & M_0+d_{12}(N) \\ N_0+d_{21}(N) & a_{22}(I_\mathcal{A})+NM_0\\\end{array}\right]\left[\begin{array}{cc} -MN & M \\  0 & 0  \end{array}\right]\\
&&~~~~~~~~~~~~~~~~~~~+ n\left[\begin{array}{cc}I_\mathcal{A} & 0 \\N & 0\\\end{array}\right]
\left[\begin{array}{cc} -a_{11}(MN)-MN_0 & -MNM_0+c_{12}(M) \\ -N_0MN+c_{21}(M) & -a_{22}(MN)+N_0M \end{array}\right].
\end{eqnarray*}
The above matrix relation implies
\begin{eqnarray}
b_{22}(NM)=Nc_{12}(M)+d_{21}(N)M+a_{22}(MN)\label{0016}
\end{eqnarray}
for all $N\in\mathcal{N}$ and $M\in\mathcal{M}$.
Similarly, by considering $S=\left[\begin{array}{cc}0 & 0 \\N & I_\mathcal{A} \\\end{array}\right]$ and
$T=\left[\begin{array}{cc}0 & M \\0 & -NM \\\end{array}\right]$, we have
\begin{eqnarray}
a_{11}(MN)=c_{12}(M)N+Md_{21}(N)+b_{11}(NM)\label{0017}
\end{eqnarray}
for all $N\in\mathcal{N}$ and $M\in\mathcal{M}$.

By (\ref{0005}), (\ref{0006}) and (\ref{00081})-(\ref{0017}), the proof is complete.
\end{proof}

Before proving the theorem, we show several propositions concerning the structure of
Lie derivations, Jordan derivations and derivations on
$\mathcal{U}$.

\begin{proposition}\label{203}
A linear mapping $\delta$ on $\mathcal{U}=\left[\begin{array}{cc}\mathcal{A} & \mathcal{M} \\ \mathcal{N} & \mathcal{B} \\ \end{array} \right]$
with $char(\mathcal{U})\neq2$ is a Lie derivation if and only if it is of the form
$$\delta\left(\left[\begin{array}{cc}A & M \\N & B \\\end{array}\right]\right)=\left[\begin{array}{cc}a_{11}(A)+b_{11}(B)-MN_0-M_0N & AM_0-M_0B+c_{12}(M) \\N_0A-BN_0+d_{21}(N) & a_{22}(A)+b_{22}(B)+N_0M+NM_0 \\\end{array}\right],$$
where $N_0\in\mathcal{N}$, $M_0\in\mathcal{M}$ and $a_{11}:\mathcal{A}\rightarrow\mathcal{A}$, $b_{22}:\mathcal{B}\rightarrow\mathcal{B}$, $b_{11}:\mathcal{B}\rightarrow Z(\mathcal{A})$, $a_{22}:\mathcal{A}\rightarrow Z(\mathcal{B})$, $c_{12}:\mathcal{M}\rightarrow\mathcal{M}$, $d_{21}:\mathcal{N}\rightarrow\mathcal{N}$ are linear mappings satisfying

$\mathrm{(1)}$~~~$a_{11}$ is a Lie derivation on  $\mathcal{A}$, $a_{22}([A,A'])=0$,
$c_{12}(AM)=a_{11}(A)M-Ma_{22}(A)+Ac_{12}(M)$ and $d_{21}(NA)=Na_{11}(A)-a_{22}(A)N+d_{21}(N)A$;

$\mathrm{(2)}$~~~$b_{22}$ is a Lie derivation on  $\mathcal{B}$, $b_{11}([B,B'])=0$,
$c_{12}(MB)=Mb_{22}(B)-b_{11}(B)M+c_{12}(M)B$ and $d_{21}(BN)=b_{22}(B)N-Nb_{11}(B)+Bd_{21}(N)$;

$\mathrm{(3)}$~~~$a_{11}(MN)=c_{12}(M)N+Md_{21}(N)+b_{11}(NM)$, $b_{22}(NM)=d_{21}(N)M+Nc_{12}(M)+a_{22}(MN).$\\
\end{proposition}
\begin{proof}
We just show the necessity, for the sufficiency can be achieved by elementary calculations.

We will consider the equation $\delta([S,T])=[\delta(S),T]+[S,\delta(T)]$ entry-wise.
Take $S=A\oplus 0$
and $T=I_\mathcal{A}\oplus 0$,
we have\\
\begin{eqnarray*}
0=\delta([S,T])=[\delta(S),T]+[S,\delta(T)]=
\left[\begin{array}{cc}Aa_{11}(I_\mathcal{A})-a_{11}(A) & Aa_{12}(I_\mathcal{A})-a_{12}(A) \\a_{21}(A)-a_{21}(I_\mathcal{A})A & 0\\\end{array}\right].
\end{eqnarray*}
So $a_{12}(A)=AM_0$ and $a_{21}(A)=N_0A$, where $M_0=a_{12}(I_\mathcal{A})$ and $N_0=a_{21}(I_\mathcal{A})$.
Similarly if we take $S=0 \oplus B$
and $T=I_\mathcal{A}\oplus 0$,
then we have $b_{12}(B)=-M_0B$ and $b_{21}(B)=-BN_0$.

For arbitrary $A,A^\prime\in \mathcal{A}$, setting $S=A \oplus 0$
and $T=A'\oplus 0$,
we obtain that $a_{11}$ is a Lie derivation on $\mathcal{A}$ and $a_{22}([A,A'])=0$.
Symmetrically, take $S=0\oplus B$
and $T=0\oplus B'$, we see that $b_{22}$ is a Lie derivation on $\mathcal{B}$ and $b_{11}([B,B'])=0$.

Taking $S=\left[\begin{array}{cc}A & 0 \\0 & 0\\\end{array}\right]$
and $T=\left[\begin{array}{cc}0 & M \\0 & 0\\\end{array}\right]$
yields $c_{11}(M)=-MN_0$, $c_{12}(AM)=a_{11}(A)M-Ma_{22}(A)+Ac_{12}(M)$, $c_{21}=0$ and $c_{22}(M)=N_0M.$
Moreover, taking $S=\left[\begin{array}{cc}0 & 0 \\N & 0\\\end{array}\right]$
and $T=\left[\begin{array}{cc}A & 0 \\0 & 0\\\end{array}\right]$
gives $d_{11}(N)=-M_0N$, $d_{12}=0$, $d_{21}(NA)=Na_{11}(A)-a_{22}(A)N+d_{21}(N)A$ and $d_{22}(N)=NM_0.$
Taking $S=\left[\begin{array}{cc}0 & M \\0 & 0\\\end{array}\right]$ and $T=\left[\begin{array}{cc}0 & 0 \\0 & B\\\end{array}\right]$ leads to $c_{12}(MB)=Mb_{22}(B)-b_{11}(B)M+c_{12}(M)B$. Taking $S=\left[\begin{array}{cc}0 & 0 \\0 & B\\\end{array}\right]$ and $T=\left[\begin{array}{cc}0 & 0 \\N & 0\\\end{array}\right]$ gives $d_{21}(BN)=b_{22}(B)N-Nb_{11}(B)+Bd_{21}(N)$. \

Furthermore, consider $S=A\oplus 0$ and $T=0\oplus B$, we obtain $[a_{22}(A),B]=0$ and $[A, b_{11}(B)]=0$. As A and B are arbitrary, we have $a_{22}(A)\in Z(\mathcal{B})$ and $b_{11}(B)\in Z(\mathcal{A})$.\

Finally, let $S=\left[\begin{array}{cc}0 & M \\0 & 0\\\end{array}\right]$ and $T=\left[\begin{array}{cc}0 & 0 \\N & 0\\\end{array}\right]$.
A simple calculation gives $a_{11}(MN)=c_{12}(M)N+Md_{21}(N)+b_{11}(NM)$ and $b_{22}(NM)=d_{21}(N)M+Nc_{12}(M)+a_{22}(MN).$
\end{proof}

With a proof similar to the proof of Proposition \ref{203}, we have the following result.

\begin{proposition}\label{204}
A linear mapping $\delta$ on $\mathcal{U}=\left[\begin{array}{cc}\mathcal{A} & \mathcal{M} \\ \mathcal{N} & \mathcal{B} \\ \end{array} \right]$
with $char(\mathcal{U})\neq 2$ is a Jordan derivation if and only if it is of the form
$$\delta\left(\left[\begin{array}{cc}A & M \\N & B \\\end{array}\right]\right)=\left[\begin{array}{cc}a_{11}(A)-MN_0-M_0N & AM_0-M_0B+c_{12}(M)+d_{12}(N) \\N_0A-BN_0+c_{21}(M)+d_{21}(N) & b_{22}(B)+N_0M+NM_0 \\\end{array}\right],$$
where $N_0\in\mathcal{N}$, $M_0\in\mathcal{M}$ and $a_{11}:\mathcal{A}\rightarrow\mathcal{A}$, $b_{22}:\mathcal{B}\rightarrow\mathcal{B}$,  $c_{12}:\mathcal{M}\rightarrow\mathcal{M}$, $d_{21}:\mathcal{N}\rightarrow\mathcal{N}$, $c_{21}:\mathcal{M}\rightarrow\mathcal{N}$, $d_{12}:\mathcal{N}\rightarrow\mathcal{M}$ are linear mappings satisfying

$\mathrm{(1)}$~~~$a_{11}$ is a Jordan derivation on  $\mathcal{A}$, $c_{12}(AM)=a_{11}(A)M+Ac_{12}(M)$,
$c_{21}(AM)=c_{21}(M)A$, $c_{12}(MB)=Mb_{22}(B)+c_{12}(M)B$, $c_{21}(MB)=Bc_{21}(M)$, $c_{21}(M)M=0$ and $Mc_{21}(M)=0$;

$\mathrm{(2)}$~~~$b_{22}$ is a Jordan derivation on  $\mathcal{B}$, $d_{21}(NA)=Na_{11}(A)+d_{21}(N)A$, $d_{12}(NA)=Ad_{12}(N)$,
$d_{21}(BN)=b_{22}(B)N+Bd_{21}(N)$, $d_{12}(BN)=d_{12}(N)B$, $d_{12}(N)N=0$ and $Nd_{12}(N)=0$;

$\mathrm{(3)}$~~~$a_{11}(MN)=c_{12}(M)N+Md_{21}(N)$, $b_{22}(NM)=d_{21}(N)M+Nc_{12}(M)$.

\end{proposition}
Since every derivation is a Lie derivation as well as a Jordan derivation, combining the propositions above yields the following corollary.

\begin{corollary}\label{205}
A linear mapping $\delta$ on $\mathcal{U}=\left[\begin{array}{cc}\mathcal{A} & \mathcal{M} \\ \mathcal{N} & \mathcal{B} \\ \end{array} \right]$
with $char(\mathcal{U})\neq 2$ is a derivation if and only if it is of the form\\
$$\delta\left(\left[\begin{array}{cc}A & M \\N & B \\\end{array}\right]\right)=\left[\begin{array}{cc}a_{11}(A)-MN_0-M_0N & AM_0-M_0B+c_{12}(M) \\N_0A-BN_0+d_{21}(N) & b_{22}(B)+N_0M+NM_0 \\\end{array}\right],$$
where $N_0\in\mathcal{N}$, $M_0\in\mathcal{M}$ and $a_{11}:\mathcal{A}\rightarrow\mathcal{A}$, $b_{22}:\mathcal{B}\rightarrow\mathcal{B}$, $c_{12}:\mathcal{M}\rightarrow\mathcal{M}$, $d_{21}:\mathcal{N}\rightarrow\mathcal{N}$ are linear mappings satisfying

$\mathrm{(1)}$~~~$a_{11}$ is a derivation on  $\mathcal{A}$, $c_{12}(AM)=a_{11}(A)M+Ac_{12}(M)$ and $d_{21}(NA)=Na_{11}(A)+d_{21}(N)A$;

$\mathrm{(2)}$~~~$b_{22}$ is a derivation on  $\mathcal{B}$, $c_{12}(MB)=Mb_{22}(B)+c_{12}(M)B$ and $d_{21}(BN)=b_{22}(B)N+Bd_{21}(N)$;

$\mathrm{(3)}$~~~$a_{11}(MN)=c_{12}(M)N+Md_{21}(N)$, $b_{22}(NM)=d_{21}(N)M+Nc_{12}(M)$.
\end{corollary}

Now we are in a position to prove our main theorem of this section.\
\leftline{}
\textbf{Proof of Theorem \ref{201}}~~~(1)~~~Assume that $(m+n)(m-n)\neq 0$.
By Lemma \ref{202} (i), we have $nc_{21}(AM)=mc_{21}(M)A$. Taking $A=I_\mathcal{A}$ gives $nc_{21}(M)=mc_{21}(M)$,
which implies $c_{21}(M)=0$ for every $M\in \mathcal{M}$.
Similarly, by $nd_{12}(NA)=mAd_{12}(N)$, we obtain $d_{12}(N)=0$ for every $N\in \mathcal{N}$.

Since $mBa_{22}(A)+na_{22}(A)B=0$, choosing $B=I_\mathcal{B}$ gives $a_{22}(A)=0$ for every $A\in \mathcal{A}$.
Hence $c_{12}(AM)=a_{11}(A)M+Ac_{12}(M)$, $d_{21}(NA)=Na_{11}(A)+d_{21}(N)A$ and $b_{22}(NM)=Nc_{12}(M)+d_{21}(N)M$
for all $A\in \mathcal{A}$, $M\in \mathcal{M}$ and $N\in \mathcal{N}$.

So for any $A_1, A_2\in \mathcal{A}$ and $M\in \mathcal{M}$,
\begin{eqnarray*}
c_{12}(A_1A_2M)&=&a_{11}(A_1A_2)M+A_1A_2c_{12}(M);\\
c_{12}(A_1A_2M)&=&a_{11}(A_1)A_2M+A_1c_{12}(A_2M)\\
&=&a_{11}(A_1)A_2M+A_1a_{11}(A_2)M+A_1A_2c_{12}(M).
\end{eqnarray*}
Thus $(a_{11}(A_1A_2)-A_1a_{11}(A_2)-a_{11}(A_1)A_2)M=0$. Since $\mathcal{M}$ is a faithful left $\mathcal{A}$-module, we have
\begin{eqnarray*}
a_{11}(A_1A_2)-A_1a_{11}(A_2)-a_{11}(A_1)A_2=0.
\end{eqnarray*}

Similarly, by Lemma \ref{202} (ii), we have $b_{11}(B)=0$ for every $B\in \mathcal{B}$. Hence
$c_{12}(MB)=Mb_{22}(B)+c_{12}(M)B$, $d_{21}(BN)=b_{22}(B)N+Bd_{21}(N)$ and $a_{11}(MN)=c_{12}(M)N+Md_{21}(N)$
for all $B\in \mathcal{B}$, $M\in \mathcal{M}$ and $N\in \mathcal{N}$.
Now combining Corollary \ref{205}, we complete the proof.\

(2)~~~Assume that $m+n=0$. Without loss of generality, we may assume that $m=1$ and $n=-1$.\

By the proof of (1), we have $c_{21}(M)=0$ for every $M\in \mathcal{M}$ and $d_{12}(N)=0$ for every $N\in \mathcal{N}$.
By Lemma \ref{202} (i), we have $Ba_{22}(A)=a_{22}(A)B$ for all $A\in \mathcal{A}$ and $B\in \mathcal{B}$,
which yields $a_{22}(A)\in \mathcal{Z}(\mathcal{B})$, the center of $\mathcal{B}$.
Since $$c_{12}(AM)=a_{11}(A)M-Ma_{22}(A)+Ac_{12}(M)$$
for all $A\in \mathcal{A}$ and $M\in \mathcal{M}$. Thus for any $A_1, A_2\in \mathcal{A}$ and $M\in \mathcal{M}$,
\begin{eqnarray*}
c_{12}(AA'M)&=&a_{11}(AA')M-Ma_{22}(AA')+AA'c_{12}(M),\\
c_{12}(AA'M)&=&a_{11}(A)A'M+Aa_{11}(A')M+AA'c_{12}(M)-AMa_{22}(A')-A'Ma_{22}(A),\\
c_{12}(A'AM)&=&a_{11}(A'A)M-Ma_{22}(A'A)+A'Ac_{12}(M),\\
c_{12}(A'AM)&=&a_{11}(A')AM+A'a_{11}(A)M+A'Ac_{12}(M)-A'Ma_{22}(A)-AMa_{22}(A'),
\end{eqnarray*}
whence
\begin{eqnarray*}
c_{12}([A,A']M)&=&([a_{11}(A),A']+[A,a_{11}(A')])M+[A,A']c_{12}(M),\\
c_{12}([A,A']M)&=&a_{11}([A,A'])M-Ma_{22}([A,A'])+[A,A']c_{12}(M).
\end{eqnarray*}
Since $\delta([\mathcal{A},\mathcal{A}])\bigcap \mathcal{B}=0$, we have $a_{22}$ vanishes on commutators,
which implies $$a_{11}([A,A'])M=([a_{11}(A),A']+[A,a_{11}(A')])M.$$
Since $\mathcal{M}$ is a faithful left $\mathcal{A}$-module, we have $a_{11}$ is a Lie derivation on $\mathcal{A}$.\

Similarly, by Lemma \ref{202} (ii), we have $b_{11}(B)\in \mathcal{Z}(\mathcal{A})$ and $b_{22}$ is a Lie derivation on $\mathcal{B}$.\

(3)~~~The proof when $m=n$ is analogous to (1) and we leave it to the readers.~~~~~~~~~~~~~~~~~~~~~~~~~~~~~~~~~~~~~~$\Box$

\begin{remark}
\emph{In Theorem \ref{201}, the assumption that $\mathcal{M}$ is a faithful ($\mathcal{A}$, $\mathcal{B}$)-bimodule
may be replaced by one of the following conditions:\\
(1)~~$\mathcal{N}$ is a faithful ($\mathcal{B}$, $\mathcal{A}$)-bimodule;\\
(2)~~$\mathcal{M}$ is a faithful left $\mathcal{A}$-module and  $\mathcal{N}$ is a faithful left $\mathcal{B}$-bimodule;\\
(3)~~$\mathcal{M}$ is a faithful right $\mathcal{B}$-module and  $\mathcal{N}$ is a faithful right $\mathcal{A}$-bimodule,\\
while the corresponding proofs need some minor modifications.}
\end{remark}

\begin{remark}
\emph{For the case $m=n$, $\delta$ may not be a derivation. For instance, let $\mathcal{A}$ and $\mathcal{B}$ be the algebras of $2\times2$
diagonal matrices over $\mathbb{C}$, $\mathcal{M}$ be the module of $2\times2$ matrices over $\mathbb{C}$ that vanishes on all but the $(1,1)$-entry, and $\mathcal{N}$ be the module of $2\times2$ matrices over $\mathbb{C}$ that vanishes on all but the $(2,2)$-entry. Let $\phi_{\mathcal{M}\mathcal{N}}$ and $\varphi_{\mathcal{N}\mathcal{M}}$ be the mappings that coincide with the usual matrix multiplication. Let
$\mathcal{U}=\left[  \begin{array}{cc}
    \mathcal{A} & \mathcal{M} \\
    \mathcal{N} & \mathcal{B} \\
  \end{array}
\right]$
be the generalized matrix algebra originated from the Morita context
$(\mathcal{A}, \mathcal{B}, \mathcal{M}, \mathcal{N}, \phi_{\mathcal{M}\mathcal{N}}, \varphi_{\mathcal{N}\mathcal{N}})$.
i.e. every element in $\mathcal{U}$ is of the form
$$\left[
  \begin{array}{cccc}
    a & 0 & m & 0 \\
    0 & a & 0 & 0 \\
    0 & 0 & b & 0 \\
    0 & n & 0 & b \\
  \end{array}
\right], $$
where $a, b, m, n\in\mathbb{C}$.
Now, let $\delta\left(\left[
  \begin{array}{cccc}
    a & 0 & m & 0 \\
    0 & a & 0 & 0 \\
    0 & 0 & b & 0 \\
    0 & n & 0 & b \\
  \end{array}
\right]\right)=
\left[
  \begin{array}{cccc}
    0 & 0 & 0 & 0 \\
    0 & 0 & 0 & 0 \\
    0 & 0 & 0 & 0 \\
    0 & m & 0 & 0 \\
  \end{array}
\right]$, then it is easy to verify that $\delta$ is a Jordan derivation but not a derivation. That is, $\delta$ is a proper Jordan derivation.}
\end{remark}

Note that a unital prime ring $\mathcal{A}$ with a non-trivial idempotent $P$ can be written as the matrix form
$\left[
  \begin{array}{cc}
    P\mathcal{A}P & P\mathcal{A}P^\perp \\
    P^\perp\mathcal{A}P & P^\perp\mathcal{A}P^\perp \\
  \end{array}
\right]$.
Moreover, for any $A\in \mathcal{A}$, $PAP\mathcal{A}(I-P)=0$ and $P\mathcal{A}(I-P)A(I-P)$ imply $PAP=0$ and $(I-P)A(I-P)=0$, respectively.
Note that every Jordan derivation of 2-torsion free prime
rings is a derivation(\cite{prime}). So the following corollary is immediate.

\begin{corollary}\label{206}
Let $m+n\neq 0$ and $\mathcal{A}$ be a unital prime ring with characteristic neither $|mn(m+n)|$ nor $|m-n|$.
Assume that $\mathcal{A}$ contains a non-trivial idempotent $P$.
If $\delta$ is an (\textit{m, n})-derivable mapping at 0 from $\mathcal{A}$ into itself, then $\delta$ is a derivation.
\end{corollary}

As von Neumann algebras have rich idempotent elements and factor von Neumann algebras are prime, the following corollary is obvious.
\begin{corollary}\label{207}
Let $\mathcal{A}$ be a factor von Neumann algebra. If $\delta$ is an (m, n)-derivable mapping at 0 from $\mathcal{A}$ into itself with $m+n\neq 0$, then $\delta$ is a derivation.
\end{corollary}

Obviously, when $\mathcal{N}=0$, $\mathcal{U}$ degenerates to an upper triangular algebra. By \cite{JDOTA},
each Jordan derivation of an upper triangular algebra
is a derivation. Thus we have the following corollary, which generalizes \cite[Theorem 2.1]{ZHUJUN}.
\begin{corollary}\label{208}
Let $\mathcal{U}=Tri(\mathcal{A},\mathcal{M},\mathcal{B})$ be an upper triangular algebra
such that $\mathcal{M}$ is a \textit{faithful} ($\mathcal{A}$,$\mathcal{B}$)-bimodule.
Let $\delta$ be an (m, n)-derivable mapping at 0 from $\mathcal{U}$ into itself.\

$\mathrm{(1) }$~~~If $(m+n)(m-n)\neq 0$ and $char(\mathcal{U})\neq |mn(m+n)(m-n)|$, then $\delta$ is a derivation.

$\mathrm{(2) }$~~~If $m-n=0$ and $char(\mathcal{U})\neq |2m|$, then $\delta$ is a derivation.

$\mathrm{(3)}$~~~If $m+n=0$,
$\delta([\mathcal{A},\mathcal{A}])\bigcap \mathcal{B}=0$, $\delta([\mathcal{B},\mathcal{B}])\bigcap \mathcal{A}=0$ and $char(\mathcal{U})\neq |2m|$,
then $\delta$ is a Lie derivation.
\end{corollary}

Let $\mathcal{N}$ be a nest on $H$ and alg$\mathcal{N}$ be the associated algebra.
If $\mathcal{N}$ is trivial, then alg$\mathcal{N}$ is $B(H)$. If $\mathcal{N}$ is nontrivial, take a nontrivial projection $P\in \mathcal{N}$.
Let $\mathcal{A}=P\mathrm{alg}\mathcal{N}P$, $\mathcal{M}=P\mathrm{alg}\mathcal{N}(I-P)$ and $\mathcal{B}=(I-P)\mathrm{alg}\mathcal{N}(I-P)$.
Then $\mathcal{M}$ is a faithful ($\mathcal{A}$, $\mathcal{B}$)-bimodule, and alg$\mathcal{N}$=Tri($\mathcal{A}$, $\mathcal{M}$, $\mathcal{B}$)
is an upper triangular algebra.
Thus as an application of Corollary \ref{207} and Corollary \ref{208}, we have the following corollary.
\begin{corollary}\label{209}
Let $\mathcal{N}$ be a nest on a Hilbert space $H$ and alg$\mathcal{N}$ be the associated algebra.
If $\delta$ is an (m, n)-derivable mapping at 0 from $alg\mathcal{N}$ into itself with $m+n\neq 0$, then $\delta$ is a derivation.
\end{corollary}

\section{(\textit{m, n})-derivable mappings at $I_\mathcal{A}\oplus 0$}

\

In this section, we study (\textit{m, n})-derivable mappings at $I_\mathcal{A}\oplus 0$.

\begin{theorem}\label{301}
Let $\mathcal{U}=\left[
  \begin{array}{cc}
    \mathcal{A} & \mathcal{M} \\
    \mathcal{N} & \mathcal{B} \\
  \end{array}
\right]$ be a generalized matrix algebra and $\delta$ be an $(m, n)$-derivable mapping at $I_\mathcal{A}\oplus 0$ from $\mathcal{U}$ into itself.
Suppose that for every $A$ in $\mathcal{A}$, there exists an integer $k$ such that $kI_\mathcal{A}+A$ is invertible in $\mathcal{A}$.
Assume that $\mathcal{M}$ is a faithful $(\mathcal{A}, \mathcal{B})$-bimodule.

$\mathrm{(1) }$~~~If $(m+n)(m-n)\neq 0$ and $char(\mathcal{U})\neq |mn(m+n)(m-n)|$, then $\delta$ is a derivation.

$\mathrm{(2)}$~~~If $m+n=0$, $\delta([\mathcal{A},\mathcal{A}])\bigcap \mathcal{B}=0$, $\delta([\mathcal{B},\mathcal{B}])\bigcap \mathcal{A}=0$ and
$char(\mathcal{U})\neq |2m|$, then $\delta$ is a Lie derivation.

$\mathrm{(3)}$~~~If $m-n=0$ and $char(\mathcal{U})\neq |2m|$, then $\delta$ is a Jordan derivation.
\end{theorem}

We proceed with the following lemma.

\begin{lemma}\label{302}
Let $\mathcal{U}=\left[
  \begin{array}{cc}
    \mathcal{A} & \mathcal{M} \\
    \mathcal{N} & \mathcal{B} \\
  \end{array}
\right]$ be a generalized matrix algebra with $char(\mathcal{U})\neq |mn|$ and $\delta$ be an $(m, n)$-derivable mapping at $I_\mathcal{A}\oplus 0$ from $\mathcal{U}$ into itself.
Suppose that for every $A$ in $\mathcal{A}$, there exists an integer $k$ such that $kI_\mathcal{A}+A$ is invertible in $\mathcal{A}$.
Then
$$\delta\left(\left[\begin{array}{cc}A & M \\N & B \\\end{array}\right]\right)=
\left[\begin{array}{cc}a_{11}(A)+b_{11}(B)-MN_0-M_0N & AM_0-M_0B+c_{12}(M)+d_{12}(N)
\\N_0A-BN_0+c_{21}(M)+d_{21}(N) & a_{22}(A)+b_{22}(B)+N_0M+NM_0 \\\end{array}\right],$$
where $N_0\in\mathcal{N}$, $M_0\in\mathcal{M}$ and $a_{11}:\mathcal{A}\rightarrow\mathcal{A}$,
$a_{22}:\mathcal{A}\rightarrow\mathcal{B}$, $b_{11}:\mathcal{B}\rightarrow\mathcal{A}$,
$b_{22}:\mathcal{B}\rightarrow\mathcal{B}$, $c_{12}:\mathcal{M}\rightarrow\mathcal{M}$,
$c_{21}:\mathcal{M}\rightarrow\mathcal{N}$, $d_{12}:\mathcal{N}\rightarrow\mathcal{M}$,
$d_{21}:\mathcal{N}\rightarrow\mathcal{N}$ are linear mappings satisfying
\begin{eqnarray*}
\mathrm{(i)}&&c_{12}(AM)=Ma_{22}(A)+a_{11}(A)M+Ac_{12}(M),~mc_{21}(AM)=nc_{21}(M)A,\\
&&d_{21}(NA)=a_{22}(A)N+Na_{11}(A)+d_{21}(N)A,~md_{12}(NA)=nAd_{12}(N),\\
&&mBa_{22}(A)+na_{22}(A)B=0;\\
\mathrm{(ii)}&&mc_{12}(MB)=nb_{11}(B)M+mMb_{22}(B)+mc_{12}(M)B,~mc_{21}(MB)=nBc_{21}(M),\\
&&md_{21}(BN)=nNb_{11}(B)+mb_{22}(B)N+mBd_{21}(N),~md_{12}(BN)=nd_{12}(N)B,\\
&&mb_{11}(B)A+nAb_{11}(B)=0;\\
\mathrm{(iii)}&&a_{11}(MN)=c_{12}(M)N+Md_{21}(N)+b_{11}(NM),~b_{22}(NM)=Nc_{12}(M)+d_{21}(N)M+a_{22}(MN).
\end{eqnarray*}
\end{lemma}

\begin{proof}
We prove the lemma by several steps.\

\textbf{Step 1:}~~
For every invertible element $A\in\mathcal{A}$ and every $B\in\mathcal{B}$,
let $S=A\oplus B$
and $T=A^{-1}\oplus 0$. Then $ST=I_\mathcal{A}\oplus 0$ and after elementary matrix computation, we have the following four equations.
\begin{eqnarray}
&&(m+n)a_{11}(I_\mathcal{A})=ma_{11}(A)A^{-1}+mb_{11}(B)A^{-1}+mAa_{11}(A^{-1})+na_{11}(A^{-1})A\nonumber\\
&&~~~~~~~~~~~~~~~~~~~~~~~~+nA^{-1}a_{11}(A)+nA^{-1}b_{11}(B),\label{3001}\\
&&(m+n)a_{12}(I_\mathcal{A})=mAa_{12}(A^{-1})+na_{12}(A^{-1})B+nA^{-1}a_{12}(A)+nA^{-1}b_{12}(B),\label{3002}\\
&&(m+n)a_{21}(I_\mathcal{A})=ma_{21}(A)A^{-1}+mb_{21}(B)A^{-1}+mBa_{21}(A^{-1})+na_{21}(A^{-1})A,\label{3003}\\
&&(m+n)a_{22}(I_\mathcal{A})=mBa_{22}(A^{-1})+na_{22}(A^{-1})B.\label{3004}
\end{eqnarray}

Taking $B=0$ in (\ref{3001}) and (\ref{3004}), and since for every element $A\in \mathcal{A}$, there exists an integer $k$ such that $kI_\mathcal{A}+A$ is invertible in $\mathcal{A}$, by calculation we have
\begin{eqnarray}
mBa_{22}(A)+na_{22}(A)B=0~~~&\textrm{and}&~~~mb_{11}(B)A+nAb_{11}(B)=0,\label{3007}
\end{eqnarray}
for every $A\in \mathcal{A}$ and every $B\in \mathcal{B}$.
Similarly, by (\ref{3002}) and (\ref{3003}), we obtain
\begin{eqnarray}
a_{12}(A)B=-Ab_{12}(B)~~~&\textrm{and}&~~~Ba_{21}(A)=-b_{21}(B)A\label{3008}
\end{eqnarray}
for every $A\in \mathcal{A}$ and every $B\in \mathcal{B}$.
Let $M_0=a_{12}(I_\mathcal{A})$ and $N_0=a_{21}(I_\mathcal{A})$. By (\ref{3008}), we have
\begin{eqnarray}
a_{12}(A)=AM_0,~b_{12}(B)=-M_0B,~a_{21}(A)=N_0A~~\textrm{and}~~b_{21}(B)=-BN_0.\label{30071}
\end{eqnarray}
for every $A\in \mathcal{A}$ and every $B\in \mathcal{B}$.

\textbf{Step 2:}~~
For every invertible element $A\in\mathcal{A}$, every $B\in\mathcal{B}$ and every $M\in\mathcal{M}$,
let $S=\left[\begin{array}{cc}A & AM \\0 & 0 \\\end{array}\right]$
and $T=\left[\begin{array}{cc}A^{-1} & -MB \\0 & B \\\end{array}\right]$. Then $ST=I_\mathcal{A}\oplus 0$ and after elementary matrix computation we have the following four equations.
\begin{eqnarray}
nc_{11}(M)=mc_{11}(AM)A^{-1}+mAb_{11}(B)-mAc_{11}(MB)+mAMN_0A^{-1}~~~~~~~~~~~~\nonumber\\
-mAMBN_0-mAMc_{21}(MB)+nb_{11}(B)A-nc_{11}(MB)A~~~~~~~~~~~~~\label{3009}\\
+nA^{-1}c_{11}(AM)-nMBN_0A-nMBc_{21}(AM),~~~~~~~~~~~~~~~~~~~~~~~~~~~\nonumber
\end{eqnarray}
\begin{eqnarray}
nc_{12}(M)=-ma_{11}(A)MB-mc_{11}(AM)MB+mc_{12}(AM)B-mAc_{12}(MB)~~~~~~~~\nonumber\\
+mAMa_{22}(A^{-1})+mAMb_{22}(B)-mAMc_{22}(MB)+na_{11}(A^{-1})AM~~~\nonumber\\
+nb_{11}(B)AM-nc_{11}(MB)AM+nA^{-1}c_{12}(AM)-nMBa_{22}(A)~~~~~~~~\label{3010}\\
-nMBc_{22}(AM),~~~~~~~~~~~~~~~~~~~~~~~~~~~~~~~~~~~~~~~~~~~~~~~~~~~~~~~~~~~~~~~~~~~~~\nonumber
\end{eqnarray}
\begin{eqnarray}
nc_{21}(M)=mc_{21}(AM)A^{-1}-nc_{21}(MB)A+nBc_{21}(AM),~~~~~~~~~~~~~~~~~~~~~~~~~~~~~~\label{3011}
\end{eqnarray}
\begin{eqnarray}
nc_{22}(M)=-mN_0AMB+ma_{22}(A)B+nN_0M-nBN_0AM+nBa_{22}(A)~~~~~~~~~~~\nonumber\\
~~~~~~~~~~~~~~~-mc_{21}(AM)MB+mc_{22}(AM)B-nc_{21}(MB)AM+nBc_{22}(AM).~~~~~\label{3012}
\end{eqnarray}

Taking $B=0$ and $A=I_\mathcal{A}$ in (\ref{3009}) and (\ref{3012}), we have
\begin{eqnarray}
c_{11}(M)=-MN_0~~~\textrm{and}~~~c_{22}(M)=N_0M\label{3013}
\end{eqnarray}
for every $M\in \mathcal{M}$.

Taking $B=0$ in (\ref{3010}) and (\ref{3011}), and since for every element $A$ in $\mathcal{A}$, there exists an integer $k$ such that $kI_\mathcal{A}+A$ is invertible in $\mathcal{A}$, by computation we have
\begin{eqnarray}
&&nc_{12}(AM)=mMa_{22}(A)+na_{11}(A)M+nAc_{12}(M),\label{3016}\\
&&~~~~~~~~~mc_{21}(AM)=nc_{21}(M)A\label{30151}
\end{eqnarray}
for every $A\in \mathcal{A}$ and every $M\in \mathcal{M}$.

Taking $A=I_\mathcal{A}$ in (\ref{3010}) and (\ref{3016}), and by (\ref{3013}), we have
\begin{eqnarray}
&&mc_{12}(MB)=mc_{12}(M)B+nb_{11}(B)M+mMb_{22}(B)\label{3018}
\end{eqnarray}
for every $B\in \mathcal{B}$ and every $M\in \mathcal{M}$.

Taking $B=I_\mathcal{B}$ in (\ref{3010}) and (\ref{3018}) gives
\begin{eqnarray}
mc_{12}(AM)=nMa_{22}(A)+ma_{11}(A)M+mAc_{12}(M)\label{3019}
\end{eqnarray}
for every $A\in \mathcal{A}$ and every $M\in \mathcal{M}$. Combining (\ref{3016}) and (\ref{3019}), we obtain
\begin{eqnarray}
c_{12}(AM)=Ma_{22}(A)+a_{11}(A)M+Ac_{12}(M).\label{30201}
\end{eqnarray}
for every $A\in \mathcal{A}$ and every $M\in \mathcal{M}$.

By (\ref{3011}), we have $c_{21}(MB)A=Bc_{21}(AM)$ and hence $mc_{21}(MB)A=mBc_{21}(AM)=nBc_{21}(M)A$. Taking $A=I_\mathcal{A}$ gives
\begin{eqnarray}
mc_{21}(MB)=nBc_{21}(M)
\end{eqnarray}
for every $B\in \mathcal{B}$ and every $M\in \mathcal{M}$.

Symmetrically, by considering $S=\left[\begin{array}{cc}A & 0 \\-BN & B \\\end{array}\right]$
and $T=\left[\begin{array}{cc}A^{-1} & 0 \\NA^{-1} & 0 \\\end{array}\right]$, we arrive at
\begin{eqnarray}
&&d_{11}(N)=-M_0N~~~\textrm{and}~~~d_{22}(N)=NM_0,\label{3020}\\
&&md_{12}(NA)=nAd_{12}(N)~~~\textrm{and}~~~md_{12}(BN)=nd_{12}(N)B,\label{3021}\\
&&d_{21}(NA)=a_{22}(A)N+Na_{11}(A)+d_{21}(N)A,\label{3022}\\
&&md_{21}(BN)=nNb_{11}(B)+mb_{22}(B)N+mBd_{21}(N).\label{3023}
\end{eqnarray}
for every $A\in \mathcal{A}$, every $B\in \mathcal{B}$ and every $M\in \mathcal{M}$.

\textbf{Step 3:}~~
For any $N\in\mathcal{N}$ and $M\in\mathcal{M}$,
let $S=\left[\begin{array}{cc}I_\mathcal{A}-MN & M \\0 & 0 \\\end{array}\right]$ and
$T=\left[\begin{array}{cc}I_\mathcal{A}& 0 \\N & 0 \\\end{array}\right]$. Then $ST=I_\mathcal{A}\oplus 0$ and elementary matrix computation yields
\begin{eqnarray}
b_{22}(NM)=d_{21}(N)M+Nc_{12}(M)+a_{22}(MN).\label{3024}
\end{eqnarray}
for every $N\in\mathcal{N}$ and $M\in\mathcal{M}$.

Symmetrically, by considering $S=\left[\begin{array}{cc}I_\mathcal{A} & M \\0 & 0 \\\end{array}\right]$ and
$T=\left[\begin{array}{cc}I_\mathcal{A}-MN & 0 \\N & 0 \\\end{array}\right]$, we arrive at
\begin{eqnarray}
a_{11}(MN)=b_{11}(NM)+c_{12}(M)N+Md_{21}(N)\label{3025}
\end{eqnarray}
for every $N\in\mathcal{N}$ and $M\in\mathcal{M}$.\

By (\ref{3007}), (\ref{30071}), (\ref{3013}), (\ref{30151}), (\ref{3018}), (\ref{30201})-(\ref{3025}), we have completed the proof.
\end{proof}

\noindent\textbf{Proof of Theorem \ref{301}}: The proof is analogous to Theorem \ref{201}, we now only refer to Lemma \ref{302}
instead of Lemma \ref{202} and we leave it to the readers.

\begin{corollary}\label{303}
Let $m+n\neq 0$ and $\mathcal{A}$ be a unital prime ring with characteristic neither $|mn(m+n)|$ nor $|m-n|$.
Assume that $\mathcal{A}$ contains a non-trivial idempotent $P$ and for every $A\in \mathcal{A}$,
there exists an integer $k$ such that $kI_\mathcal{A}+A$ is invertible in $\mathcal{A}$.
If $\delta$ is an (\textit{m, n})-derivable mapping at $P$ from $\mathcal{A}$ into itself, then $\delta$ is a derivation.
\end{corollary}

\begin{corollary}\label{304}
Let $\mathcal{A}$ be a factor von Neumann algebra and $P\in \mathcal{A}$ be a non-trivial idempotent. If $\delta$ is an (m, n)-derivable mapping at $P$ from $\mathcal{A}$ into itself and $m+n\neq 0$, then $\delta$ is a derivation.
\end{corollary}

\begin{corollary}\label{305}
Let $\mathcal{U}=Tri(\mathcal{A},\mathcal{M},\mathcal{B})$ be an upper triangular algebra
such that $\mathcal{M}$ is a \textit{faithful} ($\mathcal{A}$,$\mathcal{B}$)-bimodule.
Assume that for every $A\in \mathcal{A}$, there exists an integer $k$ such that $kI_\mathcal{A}+A$ is invertible in $\mathcal{A}$.
Let $\delta$ be an (m, n)-derivable mapping at $I_A\oplus 0$ from $\mathcal{U}$ into itself.\

$\mathrm{(1) }$~~~If $(m+n)(m-n)\neq 0$ and $char(\mathcal{U})\neq |mn(m+n)(m-n)|$, then $\delta$ is a derivation.

$\mathrm{(2) }$~~~If $m-n=0$ and $char(\mathcal{U})\neq |2m|$, then $\delta$ is a derivation.

$\mathrm{(3)}$~~~If $m+n=0$,
$\delta([\mathcal{A},\mathcal{A}])\bigcap \mathcal{B}=0$, $\delta([\mathcal{B},\mathcal{B}])\bigcap \mathcal{A}=0$ and $char(\mathcal{U})\neq |2m|$,
then $\delta$ is a Lie derivation.
\end{corollary}

\begin{corollary}\label{306}
Let $\mathcal{N}$ be a non-trivial nest on a Hilbert space $H$ and alg$\mathcal{N}$ be the associated algebra.
If $P\in \mathcal{N}$ is a non-trivial idempotent and $\delta$ is an (m, n)-derivable mapping at $P$ from $alg\mathcal{N}$ into itself with $m+n\neq 0$, then $\delta$ is a derivation.
\end{corollary}

\section{(\textit{m, n})-derivable mappings at $I$}

\

In this section, we study (\textit{m, n})-derivable mappings at $I$ and assume $m+n\neq 0$.

\begin{theorem}\label{401}
Let $\mathcal{U}=\left[
  \begin{array}{cc}
    \mathcal{A} & \mathcal{M} \\
    \mathcal{N} & \mathcal{B} \\
  \end{array}
\right]$ be a generalized matrix algebra with characteristic neither $|3mn(m+n)|$ nor $|m-n|$.
Suppose that $\frac{1}{2}I_\mathcal{A}\in \mathcal{A}$, $\frac{1}{2}I_\mathcal{B}\in \mathcal{B}$
and for every $A\in \mathcal{A}$ and every $B\in \mathcal{B}$, there exists an integer $k$ such that $kI_\mathcal{A}+A$
is invertible in $\mathcal{A}$ and $kI_\mathcal{B}+B$ is invertible in $\mathcal{B}$.
Assume that $\mathcal{M}$ is a faithful $(\mathcal{A}, \mathcal{B})$-bimudule.
If $\delta$ is an $(m, n)$-derivable mapping at $I$ from $\mathcal{U}$ into itself, then $\delta$ is a Jordan derivation.
\end{theorem}

\begin{lemma}\label{402}
Let $\mathcal{U}=\left[
  \begin{array}{cc}
    \mathcal{A} & \mathcal{M} \\
    \mathcal{N} & \mathcal{B} \\
  \end{array}
\right]$ be a generalized matrix algebra with characteristic neither $|3mn(m+n)|$ nor $|m-n|$ and $\delta$ be an $(m, n)$-derivable mapping at $I$ from $\mathcal{U}$ into itself.
Suppose that $\frac{1}{2}I_\mathcal{A}\in \mathcal{A}$, $\frac{1}{2}I_\mathcal{B}\in \mathcal{B}$
and for every $A\in \mathcal{A}$ and every $B\in \mathcal{B}$, there exists an integer $k$ such that $kI_\mathcal{A}+A$
is invertible in $\mathcal{A}$ and $kI_\mathcal{B}+B$ is invertible in $\mathcal{B}$.
Then
$$\delta\left(\left[\begin{array}{cc}A & M \\N & B \\\end{array}\right]\right)=
\left[\begin{array}{cc}a_{11}(A)-MN_0-M_0N & AM_0-M_0B+c_{12}(M)+d_{12}(N)
\\N_0A-BN_0+c_{21}(M)+d_{21}(N) & b_{22}(B)+N_0M+NM_0 \\\end{array}\right],$$
where $N_0\in\mathcal{N}$, $M_0\in\mathcal{M}$ and $a_{11}:\mathcal{A}\rightarrow\mathcal{A}$,
$b_{22}:\mathcal{B}\rightarrow\mathcal{B}$, $c_{12}:\mathcal{M}\rightarrow\mathcal{M}$,
$c_{21}:\mathcal{M}\rightarrow\mathcal{N}$, $d_{12}:\mathcal{N}\rightarrow\mathcal{M}$,
$d_{21}:\mathcal{N}\rightarrow\mathcal{N}$ are linear mappings satisfying
\begin{eqnarray*}
\mathrm{(i)}&&c_{12}(AM)=a_{11}(A)M+Ac_{12}(M),~c_{12}(MB)=Mb_{22}(B)+c_{12}(M)B,\\
&&c_{21}(AM)=c_{21}(M)A,~c_{21}(MB)=Bc_{21}(M),~c_{21}(M)M=0~and~Mc_{21}(M)=0;\\
\mathrm{(ii)}&&d_{21}(NA)=Na_{11}(A)+d_{21}(N)A,~d_{21}(BN)=b_{22}(B)N+Bd_{21}(N),\\
&&d_{12}(NA)=Ad_{12}(N),~d_{12}(BN)=d_{12}(N)B,~d_{12}(N)N=0~and~Nd_{12}(N)=0;\\
\mathrm{(iii)}&&a_{11}(MN)=c_{12}(M)N+Md_{21}(N),~b_{22}(NM)=Nc_{12}(M)+d_{21}(N)M.
\end{eqnarray*}
\end{lemma}

\begin{proof}
Since the proof is similar to Lemma \ref{302}, we will just sketch the proof.

For every invertible element $A\in\mathcal{A}$ and every invertible element $B\in\mathcal{B}$,
taking $S=A\oplus B$
and $T=A^{-1}\oplus B^{-1}$ gives
\begin{eqnarray}
&&b_{11}(B)=0~~~\textrm{and}~~~a_{11}(I_\mathcal{A})=0\label{4005}\\
&&a_{22}(A)=0~~~\textrm{and}~~~b_{22}(I_\mathcal{B})=0\label{4006}\\
&&a_{12}(A)=AM_0~~~\textrm{and}~~~b_{12}(B)=-M_0B\label{4008}\\
&&a_{21}(A)=N_0A~~~\textrm{and}~~~b_{21}(B)=-BN_0\label{4009}
\end{eqnarray}
for every $A\in \mathcal{A}$ and $B\in \mathcal{B}$, where $M_0=a_{12}(I_\mathcal{A})$ and $N_0=a_{21}(I_\mathcal{A})$.

For every invertible element $A\in\mathcal{A}$, every invertible element $B\in\mathcal{B}$ and every $M\in\mathcal{M}$,
letting $S=\left[\begin{array}{cc}A & AM \\0 & B \\\end{array}\right]$
and $T=\left[\begin{array}{cc}A^{-1} & -MB^{-1} \\0 & B^{-1} \\\end{array}\right]$ yields
\begin{eqnarray}
&&Mc_{21}(M)=0~~~\textrm{and}~~~c_{21}(M)M=0\label{4015}\\
&&c_{11}(M)=-MN_0~~~\textrm{and}~~~c_{22}(M)=N_0M\label{4016}\\
&&c_{12}(AM)=a_{11}(A)M+Ac_{12}(M)\label{4017}\\
&&c_{12}(MB)=Mb_{22}(B)+c_{12}(M)B\label{4018}\\
&&c_{21}(AM)=c_{21}(M)A~~~\textrm{and}~~~c_{21}(MB)=Bc_{21}(M)
\end{eqnarray}
for every $A\in \mathcal{A}$, $B\in \mathcal{B}$ and every $M\in \mathcal{M}$.

Symmetrically, by considering $S=\left[\begin{array}{cc}A & 0 \\NA & B \\\end{array}\right]$
and $T=\left[\begin{array}{cc}A^{-1} & 0 \\-B^{-1}N & B^{-1} \\\end{array}\right]$, we arrive at
\begin{eqnarray}
&&d_{11}(N)=-M_0N~~~\mathrm{and}~~~d_{22}(N)=NM_0,\\
&&d_{12}(N)N=0~~~\mathrm{and}~~~Nd_{12}(N)=0,\\
&&d_{12}(NA)=Ad_{12}(N)~~~\mathrm{and}~~~d_{12}(BN)=d_{12}(N)B,\\
&&d_{21}(NA)=Na_{11}(A)+d_{21}(N)A,\label{4020}\\
&&d_{21}(BN)=b_{22}(B)N+Bd_{21}(N),\label{4022}
\end{eqnarray}
for all $A\in \mathcal{A}$, $B\in \mathcal{B}$, $M\in \mathcal{M}$ and $N\in \mathcal{N}$.

For any $N\in\mathcal{N}$ and $M\in\mathcal{M}$,
setting $S=\left[\begin{array}{cc}-I_\mathcal{A}-MN & -M \\N & I_\mathcal{B} \\\end{array}\right]$ and
$T=\left[\begin{array}{cc}-I_\mathcal{A}& -M \\N &I_\mathcal{B}+NM \\\end{array}\right]$ leads to
\begin{eqnarray}
a_{11}(MN)=c_{12}(M)N+Md_{21}(N)~~~\textrm{and}~~~b_{22}(NM)=Nc_{12}(M)+d_{21}(N)M\label{4033}
\end{eqnarray}
for all $M\in \mathcal{M}$ and $N\in \mathcal{N}$.\

By (\ref{4005})-(\ref{4033}), the proof is complete.
\end{proof}

\noindent\textbf{Proof of Theorem \ref{401}}: The proof is analogous to Theorem \ref{201}, we now only refer to Lemma \ref{402}
instead of Lemma \ref{202} and we leave it to the readers.

\begin{corollary}
Let $\mathcal{A}$ be a unital prime ring of characteristic neither $|3mn(m+n)|$ nor $|m-n|$.
Assume that $\mathcal{A}$ contains $\frac{1}{2}I$ and a non-trivial idempotent $P$, and
for every $A\in \mathcal{A}$, there exists an integer $k$ such that $kI_\mathcal{A}+A$ is invertible in $\mathcal{A}$.
If $\delta$ is an (\textit{m, n})-derivable mapping at $I$ from $\mathcal{A}$ into itself, then $\delta$ is a Jordan derivation.
\end{corollary}

\begin{corollary}
Let $\mathcal{A}$ be a factor von Neumann algebra. If $\delta$ is an (m, n)-derivable mapping at $I$ from $\mathcal{A}$ into itself, then $\delta$ is a Jordan derivation.
\end{corollary}

\begin{corollary}
Let $\mathcal{U}=Tri(\mathcal{A},\mathcal{M},\mathcal{B})$ be an upper triangular algebra
such that $\mathcal{M}$ is a \textit{faithful} ($\mathcal{A}$,$\mathcal{B}$)-bimodule.
Assume that $\frac{1}{2}I_\mathcal{A}\in \mathcal{A}$, $\frac{1}{2}I_\mathcal{B}\in \mathcal{B}$, and for every $A\in \mathcal{A}$, every $B\in \mathcal{B}$,
there exists an integer $k$ such that $kI_\mathcal{A}+A$ is invertible in $\mathcal{A}$ and $kI_\mathcal{B}+B$ is invertible in $\mathcal{B}$.
Let $\delta$ be an (m, n)-derivable mapping at $I$ from $\mathcal{U}$ into itself.\

$\mathrm{(1) }$~~~If $(m+n)(m-n)\neq 0$ and $char(\mathcal{U})\neq |3mn(m+n)(m-n)|$, then $\delta$ is a derivation.

$\mathrm{(2) }$~~~If $m-n=0$ and $char(\mathcal{U})\neq |6m|$, then $\delta$ is a derivation.
\end{corollary}

\begin{corollary}
Let $\mathcal{N}$ be a nest on a Hilbert space $H$ and alg$\mathcal{N}$ be the associated algebra.
If $\delta$ is an (m, n)-derivable mapping at $I$ from $alg\mathcal{N}$ into itself, then $\delta$ is a derivation.
\end{corollary}

\section{(\textit{m, n})-derivable mappings at 0 of CSL algebras}
\

In this section, we study (\textit{m, n})-derivable mappings at 0 on CSL algebras. Assume that $m+n\neq 0$.
We proceed with the following lemmas.

\begin{lemma}\label{501}
Let $\mathcal{A}$ be a unital algebra with $char(\mathcal{A})\neq |m+n|$. If $\delta$ is an (\textit{m, n})-derivable mapping at 0
from $\mathcal{A}$ into itself and $\delta(I)=0$,
then for each idempotent $P\in \mathcal{A}$, $\delta(P)=\delta(P)P+P\delta(P).$
\end{lemma}

\begin{proof}
For any idempotent $P\in \mathcal{A}$, $P(I-P)=0$. Then we have
\begin{eqnarray*}
0&=&m\delta(P(I-P))+n\delta((I-P)P)\\
&=&m\delta(P)(I-P)+mP\delta(I-P)+n\delta(I-P)P+n(I-P)\delta(P)\\
&=&(m+n)\delta(P)-(m+n)\delta(P)P-(m+n)P\delta(P).
\end{eqnarray*}
Thus
$\delta(P)=\delta(P)P+P\delta(P).$
\end{proof}

\begin{lemma}\label{502}
Let $\mathcal{A}$ and $\delta$ be as in Lemma \ref{501} and $\delta(I)=0$.
Then for each idempotent $P\in \mathcal{A}$ and every element $A\in \mathcal{A}$, we have\\
$\mathrm{(i)}~~ \delta(PA+AP)=\delta(P)A+P\delta(A)+\delta(A)P+A\delta(P);$\\
$\mathrm{(ii)}~~ \delta(PAP)=\delta(P)AP+P\delta(A)P+PA\delta(P).$
\end{lemma}

\begin{proof}
(i)~~~For any idempotent $P\in \mathcal{A}$, $P(I-P)A=(I-P)PA=0$. Then we have
\begin{eqnarray*}
&&m\delta(P(I-P)A)+n\delta((I-P)AP)\\
&&~~~~~~=m\delta(P)(I-P)A+mP\delta((I-P)A)+n\delta((I-P)A)P+n(I-P)A\delta(P)\\
&&~~~~~~=m\delta(P)A-m\delta(P)PA+mP\delta(A)-mP\delta(PA)+n\delta(A)P-n\delta(PA)P+nA\delta(P)-nPA\delta(P),
\end{eqnarray*}
and
\begin{eqnarray*}
&&m\delta((I-P)PA)+n\delta(PA(I-P))\\
&&~~~~~~=m\delta(I-P)PA+m(I-P)\delta(PA)+n\delta(PA)(I-P)+nPA\delta(I-P)\\
&&~~~~~~=(m+n)\delta(PA)-m\delta(P)PA-mP\delta(PA)-n\delta(PA)P-nPA\delta(P).
\leftline{}\end{eqnarray*}
Combining the two equations above gives
\begin{eqnarray}
m\delta(PA)+n\delta(AP)=m\delta(P)A+mP\delta(A)+n\delta(A)P+nA\delta(P).\label{50001}
\end{eqnarray}

Since $AP(I-P)=A(I-P)P=0$, with a similar proof of (\ref{50001}), we have
\begin{eqnarray}
m\delta(AP)+n\delta(PA)=n\delta(P)A+nP\delta(A)+m\delta(A)P+mA\delta(P).\label{50002}
\end{eqnarray}
Combining (\ref{50001}) and (\ref{50002}) yields $$\delta(PA+AP)=\delta(P)A+P\delta(A)+\delta(A)P+A\delta(P).$$
(ii)~~~Replacing $A$ by $PA+AP$ in (i), we have
\begin{eqnarray*}
&&\delta(P(PA+AP)+(PA+AP)P)\\
&&~~~~~=\delta(P)(PA+AP)+P\delta(PA+AP)+\delta(PA+AP)P+(PA+AP)\delta(P)\\
&&~~~~~=\delta(P)PA+2\delta(P)AP+P\delta(P)A+P\delta(A)+2P\delta(A)P+2PA\delta(P)+\delta(A)P\\
&&~~~~~~~~+A\delta(P)P+AP\delta(P)\\
&&~~~~~=2\delta(P)AP+2P\delta(A)P+2PA\delta(P)+\delta(P)A+P\delta(A)+\delta(A)P+A\delta(P)\\
&&~~~~~=2\delta(P)AP+2P\delta(A)P+2PA\delta(P)+\delta(PA+AP).
\end{eqnarray*}
Thus $$\delta(PAP)=\delta(P)AP+P\delta(A)P+PA\delta(P).$$
\end{proof}

\begin{corollary}\label{503}
Let $\mathcal{A}$ and $\delta$ be as in Lemma \ref{502} with $\delta(I)=0$. Suppose $\mathcal{B}$ is the subalgebra of $\mathcal{A}$
generated by all idempotents in $\mathcal{A}$. Then for any $T\in \mathcal{B}$ and any $A\in \mathcal{A}$,
$\delta(TA+AT)=\delta(T)A+T\delta(A)+\delta(A)T+A\delta(T)$.
\end{corollary}

\begin{lemma}\label{504}
Let $\mathcal{L}$ be a CSL on $H$.
If $\delta$ is an (\textit{m, n})-derivable mapping at 0 from alg$\mathcal{L}$ into itself and $\delta(I)=0$,
then for all $S,T\in alg\mathcal{L}$ and $P\in \mathcal{L}$,\\
$\mathrm{(i)}$~~ $\delta(SPT(I-P))=\delta(S)PT(I-P)+S\delta(PT(I-P));$\\
$\mathrm{(ii)}$~~ $\delta(PS(I-P)T)=\delta(PS(I-P))T+PS(I-P)\delta(T).$
\end{lemma}

\begin{proof}
(i)~~~Let $P$ be in $\mathcal{L}$. Since $\delta(P)=\delta(P)P+P\delta(P)$, we see that
$P\delta(P)P=(I-P)\delta(P)(I-P)=0$. So $\delta(P)=P\delta(P)(I-P)$. Thus by Lemma \ref{502}, for every $T\in alg\mathcal{L}$,
\begin{eqnarray*}
\delta(PT(I-P))&=&\delta(PPT(I-P)+PT(I-P)P)\\
&=&\delta(P)PT(I-P)+P\delta(PT(I-P))\\
&&+\delta(PT(I-P))P+PT(I-P)\delta(P)\\
&=&\delta(PT(I-P))P+P\delta(PT(I-P)).
\end{eqnarray*}
This implies that $\delta(PT(I-P))=P\delta(PT(I-P))(I-P)$ for every $T\in alg\mathcal{L}$.
By Lemma \ref{502}(ii), we have $(I-P)\delta(PTP)=\delta((I-P)T(I-P))P=0$ for every $T\in alg\mathcal{L}$.\

Since $PT(I-P)=P-(P-PT(I-P))$ and $P-PT(I-P)$ is an idempotent, by Corollary \ref{503}, for $S, T\in alg\mathcal{L}$,
\begin{eqnarray*}
\delta(SPT(I-P))&=&\delta(PSPPT(I-P)+PT(I-P)PSP)\\
&=&\delta(PSP)PT(I-P)+PSP\delta(PT(I-P))\\
&&+\delta(PT(I-P))PSP+PT(I-P)\delta(PSP)\\
&=&(\delta(P)SP+P\delta(S)P+PS\delta(P))PT(I-P)+PSP\delta(PT(I-P))\\
&=&\delta(S)PT(I-P)+S\delta(PT(I-P)).
\end{eqnarray*}
With a proof similar to the proof of (i), we may show that (ii) is also true.
\end{proof}

By Lemmas \ref{501}, \ref{502} and \ref{504}, with a proof analogous to \cite[Theorem, 3.2]{JDCSL}, we can obtain the following theorem.
\begin{theorem}\label{505}
Let $\mathcal{L}$ be a CSL on $H$.
If $\delta$ is a norm-continuous linear (\textit{m, n})-derivable mapping at 0 from $\mathcal{A}$ into itself and $\delta(I)=0$,
then $\delta$ is a derivation.
\end{theorem}

\section*{Acknowledgement}
The authors are thankful to the referee for careful reading of the paper and valuable suggestions.\\
This work is supported by NSF of China.

\end{document}